\documentclass[11pt]{amsart}
\usepackage{bbm,stmaryrd}
\usepackage{amssymb}
\usepackage[all]{xy}
\usepackage[pdftex]{hyperref}

\newtheorem{theorem}{Theorem}[section]
\newtheorem{lemma}[theorem]{Lemma}
\newtheorem{prop}[theorem]{Proposition}
\newtheorem{corol}[theorem]{Corollary}

\theoremstyle{definition}
\newtheorem{defin}[theorem]{Definition}
\newtheorem{exam}[theorem]{Example}

\theoremstyle{remark}
\newtheorem{remk}[theorem]{Remark}

\numberwithin{equation}{section}

\def\La{\Lambda}      
       \def\De{\Delta}
     
\def\Om{\Omega}       \def\th{\theta}
\def\al{\alpha}       
\def\be{\beta}        \def\eps{\varepsilon}
\def\ga{\gamma}       
\def\la{\lambda}      
\def\vi{\varphi}      \def\io{\iota}
\def\si{\sigma}       \def\om{\omega}

 \def\mN{\mathbb N}

 \def\mZ{\mathbb Z}

\def\aK{\mathbbm k}

\def\gC{\mathfrak c} \def\gP{\mathfrak p}
\def\gD{\mathfrak d} 
 \def\gR{\mathfrak r}

\def\gM{\mathfrak m}

 \def\dP{\mathfrak P}

 \def\dS{\mathfrak S}

\def\dI{\mathfrak I}

 \def\kN{\mathcal N}
 
\def\kC{\mathcal C}

 \def\qR{{\boldsymbol R}}

\def\mt{\mbox{-}}			
\def\*{\otimes}		\def\+{\oplus}		\def\bop{\bigoplus}
\def\sb{\subset}         \def\sp{\supset}
\def\spe{\supseteq}      \def\sbe{\subseteq}
\def\mps{\mapsto}		\def\op{^\mathrm{op}}
\def\ch{^\vee}				\def\b1{\mathbf{1}}
\def\8{\infty}					\def\bap{\bigcap}
\def\xarr{\xrightarrow}		\def\Arr{\Rightarrow}
\def\Equ{\Leftrightarrow}
\def\nea{\nearrow}		\def\sea{\searrow}
\def\niso{\not\simeq}		\def\Sb{\subsetplus}
\def\nSb{\not\hspace*{-1pt}\subsetplus}
\def\ula#1{\underline{#1}}
\def\ola#1{\overline{#1}}

\def\hH{\hat{H}}

\def\ito{\stackrel{\sim}{\to}}
\def\set#1{\left\{#1\right\}}
\def\setsuch#1#2{\left\{#1\mid #2\right\}}
\def\mtr#1{\begin{pmatrix}#1\end{pmatrix}}
\def\smtr#1{\left(\begin{smallmatrix}#1\end{smallmatrix}\right)}
\def\gnr#1{\langle #1 \rangle}

\def\lst#1#2{ #1_1 , #1_2 , \dots , #1_{#2} }
\def\lsto#1#2{ #1_0 , #1_1 , \dots , #1_{#2} }

\def\md{\mbox{-}\mathrm{mod}}
\def\ind{\mbox{-}\mathrm{ind}}
\def\lat{\mathop{\mbox{-}\mathrm{lat}}\nolimits}
\def\id{\mathrm{Id}}

\def\ann{\mathop\mathrm{Ann}\nolimits}
\def\ext{\mathop\mathrm{Ext}\nolimits}
\def\hom{\mathop\mathrm{Hom}\nolimits}
\def\CM{\mathop\mathrm{CM}\nolimits}
\def\idim{\mathop\mathrm{inj.dim}\nolimits}

\def\End{\mathop\mathrm{End}\nolimits}
\def\rad{\mathop\mathrm{rad}\nolimits}
\def\im{\mathop\mathrm{Im}\nolimits}
\def\cok{\mathop\mathrm{Cok}\nolimits}
\def\ker{\mathop\mathrm{Ker}\nolimits}
\def\uhom{\mathop{\underline{\mathrm{Hom}}}\nolimits}
\def\ohom{\mathop{\overline{\mathrm{Hom}}}\nolimits}
\def\Mat{\mathop\mathrm{Mat}\nolimits}
\def\tr{\mathop\mathrm{Tr}\nolimits}
\def\tor{\mathop\mathrm{Tor}\nolimits}
\def\Rad{\mathop\mathrm{Rad}\nolimits}
\def\wh{\mathop\mathrm{wd}\nolimits}

\def\AR{\mathrm{AR}}
\def\aut{\mathop\mathrm{Aut}\nolimits}

\def\cm{Cohen--Macaulay}
\def\iff{if and only if }
\def\feq{The following conditions are equivalent:}
\def\oc{one-to-one correspondence }
\def\fg{finitely generated}
\def\dvr{discrete valuation ring}
\def\asp{almost split}
\def\art{Aus\-lan\-der--Reiten }

\begin{document}
\title{Rejection lemma and almost split sequences}
\author{Yuriy A. Drozd}
\address{Institute of Mathematics, National Academy of Sciences of Ukraine,
Tereschenkivska 3, Kyiv 01001, Ukraine}
\email{y.a.drozd@gmail.com}
\urladdr{http://www.imath.kiev.ua/$\sim$drozd}



\dedicatory{To the memory of Volodya Kirichenko}

\begin{abstract}
We study the behaviour of almost split sequences and \art quivers of an order under rejection of bijective modules as defined in \cite{qb}. In particular,
we establish relations of stable categories and almost split sequences for an order $A$ and the order $A'$ obtained from $A$ by such rejection.
These results are specified for Gorenstein and Frobenius cases.
\end{abstract}

\maketitle

\tableofcontents

\section{Introduction}
 
 Bijective modules and ``rejection lemma'' \cite{qb} play an important role in the theory of orders and lattices, as well as Gorenstein (that is, self-bijective)
 orders (see, for instance, \cite{qb,dk1,hin1,hin2}). On the other hand, now the importance of \asp\ sequences and \art quivers is doubtful. 
 In this paper we consider the behaviour of almost split sequences and \art quivers under rejection
 of bijective modules. Namely, in Section~\ref{du} we recall general facts on orders, lattices and duality. Our considerations are a bit more general, since
 the basic commutative ring is not necessarily discrete valuation ring, though in fact all main results from the ``classical'' theory, as in \cite{cr},
 remain valid. In Section~\ref{bi} we introduce bijective lattices, rejection lemma and Gorenstein orders and establish some basic results about them.
 In particular, we find out which lattices become projective and injective after rejection (Theorem~\ref{bi6}). Section~\ref{ba} is devoted to \emph{Bass
 orders}, i.e. such that all their overrings are Gorenstein. The main result here is Theorem~\ref{ba3}, which is a substantial generalization of the criterion
 for an order to be Bass from \cite{qb}. In Section~\ref{st} we consider stable categories and relate the stable category of an order $A$ to that of the
 order $A'$ obtained by the rejection of a bijective module (Theorem~\ref{st3}). In Section~\ref{ar} we study \asp\ sequences and find out how \asp\
 sequences over $A$ can be described in terms of $A'$-lattices (Proposition~\ref{ar3} and Theorem~\ref{ar4}). Finally, in Section~\ref{go} we specify
 the preceding results for Gorenstein and Frobenius cases.
 
 \medskip\noindent
 \textsf{This paper is devoted to the memory of my friend, colleague and co-author Vladimir Kirichenko.}
 
 \section{Orders, lattices and duality} 
\label{du}
 
 \noindent
  We denote by $mM$ the direct sum of $m$ copies of a module $M$. 
  The formulae $M\sp N$ and $N\sb M$ mean that $N$ is a \emph{proper subset} of $M$.
  
 In what follows $R$ is a \emph{complete local reduced noetherian ring of Krull dimension $1$} with the maximal ideal $\gM$, the residue field
 $\aK=R/\gM$ and the total ring of fractions $K$. It follows from \cite{bh} that the ring $A$ is \cm. We denote by $R\md$ the category
 of \fg\ $R$-modules and by $R\lat$ its full subcategory of \emph{$R$-lattices}, that is of \emph{torsion free} $R$-modules or such modules 
 $M$ that the natural map $M\to K\*_RM$ is an embedding. We write $KM$ instead of $K\*_RM$ and identify $M$ with $1\*M\sbe KM$. 
 In this case $R$-lattices coincide with \emph{maximal \cm\ $R$-modules}. 
 As $R$ is complete, it has a \emph{canonical module} \cite[Corollary~3.3.8]{bh}, that is an $R$-lattice $\om_R$ such that 
 $\idim_R\om_R=1$ and $\ext^1_\qR(\aK,\om_R)=\aK$. The functor $D:M\mps\hom_R(M,\om_R)$ is an \emph{exact duality} 
 on the category $R\lat$ \cite[Theorem~3.3.10]{bh}. It means that if $0\to N\xarr\al M\xarr\be L\to 0$ 
 is an exact sequence of lattices, the sequence $0\to D  L\xarr{D \be} D  M\xarr{D \al} D  N\to0$ 
 is also exact, and the natural map $M\to DDM$ is an isomorphism. As 
 $\End_R (\om_R)\simeq\End_R R\simeq R$ and $\End_KKM\simeq K\End_RM$ for every \cm\ module $M$, 
 we have that $K\om_R\simeq K$ and we identify $\om_R$ with its image in $K$. Note also that $K$ is a direct product of fields 
 $K=\prod_{i=1}^sK_i$, where $K_i$ is the field of fractions of the ring $R/\gP_i$ and $\gP_i$ runs through minimal prime ideals of $R$. 
 
 An \emph{$R$-order} is, by definition, a semiprime $R$-algebra $A$ which is an $R$-lattice. Recall that semiprime 
 means that $A$ has no nilpotent ideals. Then $KA$ is a semisimple $K$-algebra and they say that $A$ is \emph{an $R$-order in $KA$}. 
 We denote by $Z(A)$ the center of $A$ and call
 $A$ \emph{central} if the natural map $R\to Z(A)$ is an isomorphism. If $A$ is \emph{connected}, i.e. does not decomposes as a ring, 
 its center is local and vice versa. We denote by $A\md$ the category of \fg\ $R$-modules and by $A\lat$ its full subcategory of 
 \emph{$A$-lattices}, i.e. (left) $A$-modules which are $R$-lattices. The restriction of the duality functor $D$ onto $A\lat$ gives 
 an exact duality of $A\lat$ onto $A\op\lat$, which we consider as the category of right $A$-lattices. Set $\om_A=\hom_R(A,\om_R)$. 
 It is an $A$-bimodule and, for any $A$-lattice $M$ (left or right), its dual $DM$ is identified with $\hom_A(M,\om_A)$. 
 We say \emph{finite module} instead of \emph{module of finite length} and denote by $\ell_A(M)$ the length of such module.
 We call the \emph{width} of a lattice $M$ and denote it by $\wh_A(M)$ the length $\ell_{KA}(KM)$. One easily sees that $\wh_A(M)$ is the
 maximal integer $m$ such that $M$ contains a direct sum of $m$ nonzero submodules, or, equivalently, contains a chain of submodules
 $M=M_0\sp M_1\sp\dots\sp M_m$ such that $M_i/M_{i+1}$ is a lattice for $0\le i<m$. A lattice $M$ of width $1$ is called
 \emph{L-irreducible}.\!%
 \footnote{\,One often call such a lattice \emph{irreducible}, but we will use this word in another situation.}
 %
    If the ring $R$ is fixed, we often say \emph{order} instead of $R$\emph{-order}.
   
  As the ring $R$ is complete, any finite $R$-algebra (i.e. \fg\ as $R$-module) is semiperfect \cite{lam}. Therefore, the category of
  \fg\ modules over a finite $R$-algebra $A$ is Krull--Schmidt. In particular, any \fg\ projective $A$-module is isomorphic to a direct
  summand of $A$ and there is a \oc\ between indecomposable \fg\ projective $A$-modules (called \emph{principal $A$-modules}) 
  and simple $A$-modules, which maps a principle $A$-module $P$ to $P/\gR P$, where $\gR=\rad A$. For every \fg\ $A$-module $M$
  there is an epimorphism $\pi:P\to M$ with projective $P$ and $\ker\pi\sbe\gR P$. The module $P$ is unique up to isomorphism. It is called
  the \emph{projective cover} of $M$ and denoted by $P_A(M)$. Sometimes the epimorphism $\pi$ is also called a projective cover of $M$,
  though it is only defined up to an automorphism of $P$. Obviously, $\pi$ induces an isomorphism $P/\gR P\ito M/\gR M$.
  
 An \emph{overring} of an $R$-order $A$ is an $R$-order $A'$ such that $A\sb A'\sb KA$. Then $A'/A$ is a finite module
 and $A\lat$ is a full subcategory of $A'\lat$. 
 An order is said to be \emph{maximal} if it has no overrings.
 An overring of $A$ which is a maximal order is called a \emph{maximal overring} of $A$. 
 An \emph{overmodule} of an $A$-lattice $M$ is an $A$-lattice $M'$ such that $M\sb M'\sb KM$. If $A'$ is an overring of $A$
 and $M$ is a $A$-lattice considered as a submodule in $KM$, then the $A'$-module $A'M$ is defined and is an overmodule of $M$.
 
 The following fact seems to be well-known. If $R$ is a \dvr, it is proved in \cite{cr}. The general case easily reduces to this one,
 though we have not found any source in the literature.
  
  \begin{prop}\label{du1} 
  Every $R$-order $A$ has a maximal overring. The center of a maximal order is a product of \dvr{s}. A connected maximal order
  has a unique indecomposable lattice (up to isomorphism). Conversely, if an order $A$ has a unique indecomposable lattice,
  it is maximal.
 \end{prop}
 \begin{proof}
   We may suppose $A$ connected. Then its center $Z(A)$ is also local and complete. Every overring of $A$ is a $Z(A)$-order, so
   we may suppose that $Z(A)=R$. Then $Z(KA)=K$. Let $S$ be the integral closure of $R$ in $K$. As $R$ is local and complete,
   it is an \emph{excellent ring} \cite{ma}. In particular, $S$ is a \fg\ $R$-module. As it is integrally closed, it is a direct product of \dvr{s}. The ring
   $SA$ is an $S$-order and an overring of $A$. It splits into a direct product of orders whose centers are \dvr{s}. Then \cite[Theorem~26.5]{cr}
   inplies that $SA$, hence also $A$, has a maximal overring $A'$ and $Z(A')=S$. Now the remaining assertions also follow from \cite{cr}.
 \end{proof}
 
 As the algebra $K A$ is semisimple, every \fg\ $K A$-module embeds into a \fg\ free module. It easily implies that any $A$-lattice $M$ embeds 
 into a free $A$-module. Thus $A$-lattices are just submodules of free modules (\emph{torsionless modules} in the sense of Bass \cite{ba}).

 \begin{prop}\label{du2} 
 Let $I\in A\lat$. \feq
 \begin{enumerate}
 \item  $\idim_A I=1$.
 \item  $\ext^1_A(M,I)=0$ for any $M\in A\lat$.
 \item  $\ext^i_A(M,I)=0$ for any $M\in A\lat$ and any $i\ge1$.
 \item  Any exact sequence $0\to I\to N\to M\to0$, where $M\in A\lat$ splits.
 \item  $I\simeq DP$, where $P$ is a \fg\ projective $A\op$-module.
 \item  $I$ is a direct summand of $m\om_A$ for some $m$.
 \end{enumerate}
 \emph{We call a lattice $I$ satisfying these conditions} L-injective. \emph{If an L-injective lattice is indecomposable, we call it} coprincipal.
 \end{prop}
 \begin{proof}
  $(3)\Arr(2)$ and $(2)\Equ(4)$ are obvious. 
  
  $(2)\Arr(3)$ since in a projective resolution
  \[
   \dots \to P_n\xarr{d_n} P_{n-1}\xarr{d_{n-1}} \dots\to P_2\xarr{d_2} P_1\xarr{d_1} P_0\to M\to 0
  \] 
  of $M$ all modules $M_i=\im d_i$ are lattices and $\ext^i_A(M,I)\simeq\ext^1_A(M_{i-1},I)$ for $i>1$.
  
  $(4)\Arr(5)$. By duality, $(4)$ means that every exact sequence $0\to M\to N\to D I\to 0$ splits. As there is such a sequence with projective $N$,
  it implies that $P=DI$ is projective and $I\simeq DP$.
  
  $(5)\Arr(6)$. Since a projective module $P$ is a direct summand of a free module $mA$, the module $I=DP$ is a direct summand of
  $D(mA)=m\om_A$.
  
  $(6)\Arr(2)$. Let $M$ be an $A$-lattice. Consider an exact sequence $0\to N\to P\to M\to 0$ with projective $P$. As all these modules are lattices, 
  the induced sequence
  \[
   0\to \hom_A(M,\om_A)\to \hom_A(P,\om_A)\to \hom_A(N,\om_A)\to 0
  \]
  is also exact, whence $\ext^1_A(M,\om_A)=0$. Therefore, the same holds for $m\om_A$ and for its direct summand $I$.
  
  $(3)\Equ(1)$. It is known that
  \begin{align*}
    \idim I&=\sup\setsuch{i}{\ext^i_A(A/L,I)\ne0 \text{ for some left ideal } L}=\\
    			&=\sup\setsuch{i}{\ext^{i-1}_A(L,I)\ne0 \text{ for some left ideal } L}.
  \end{align*} 
   As any ideal is a lattice, $(3)\Arr(1)$. On the contrary, if (1) holds and $M$ is a lattice, embed it into a projective module $P$. Then
   $\ext^i_A(M,I)=\ext^{i+1}_A(P/M,I)=0$ if $i\ge1$, so (3) holds.
 \end{proof}
 \noindent
 The category $A\lat$ becomes an \emph{exact category} \cite{kel} if we consider as \emph{exact pairs} (or \emph{conflations}) 
 usual short exact sequences, i.e. all triples $N\xarr\al M\xarr\be L$, where $\al=\ker\be$ and $\be=\cok\al$. Therefore, \emph{deflations} are 
 epimorphisms and \emph{inflations} are monomorphisms with torsion free cokernels (we will often use this term). 
 This exact category contains enough projectives and injectives. 
 Namely, projectives are just usual \fg\ projective $A$-modules, while injectives are L-injective lattices. To obtain a conflation
 $M\to I\to N$ with L-injective $M$, one just has to dualize an exact sequence $0\to L\to P\to D M\to0$ with projective $P$. 
 
 For lattices $M,N$ we write $M\sea N$ (respectively, $N\nea M$) if there is an epimorphism $rM\to N$ (respectively, an inflation
 $N\to rM$) for some $r$. For instance, $A\sea M$ and, dually, $M\nea\om_A$ for any lattice $M$. 
 We write $N\Sb M$ if $N$ is a direct summand of $rM$ for some $r$ 
 and $M\bowtie N$ if both $M\Sb N$ and $N\Sb M$. Recall that $\CM(A)$ is a Krull--Schmidt category. So, if $N$ is indecomposable, $N\Sb M$ 
 means that $N$ is a direct summand of $M$ and $M\bowtie N$ means that $M$ and $N$ have the same set of indecomposable direct summands.
 Note that the relations $\sea\,,\nea$ and $\Sb$ are transitive. 
 
    \begin{defin}\label{du3} 
    Let $M$ be an $A$-lattices, $E=\End_AM$ and $O(M)=\End_EM$. If the natural map $A\to O(M)$ is an isomorphism, we say that $M$
    is a \emph{strict} $A$-lattice. Obviously, then $M$ is faithful.
   \end{defin}
   \noindent
      Certainly, $O(M)$ is an overring of $A/\ann_AM$. Due to the Burnside density theorem \cite[Theorem~2.6.7]{dk}, then $O(M)$ can be 
   identified with the subset $\setsuch{a\in KA/\ann KM}{aM\sbe M}$. In particular, a faithful $A$-lattice $M$ is strict \iff 
   $\setsuch{a\in KA}{aM\sbe M}=A$. If $M$ and $N$ are faithful and $M\sea N$ or $N\nea M$, then $O(N)\spe O(M)$.
   
   \begin{prop}\label{du4} 
    For every $A$-lattice $M$ there is an exact sequence 
    \begin{equation}\label{edu1} 
    0\to O(M)\to nM\to mM
    \end{equation}
    for some $m,n$. In particular, $M$ is strict \iff there is an exact sequence
    \begin{equation}\label{edu2} 
    0\to A\xarr\al nM\xarr\be mM,
    \end{equation}   
    that is  $A\nea M$.
   \end{prop}
   \begin{proof}
    If $E=\End_AM$, there is an exact sequence of $E$-modules $mE\to nE\to M\to 0$. Applying the functor $\hom_E(\_\,,M)$, we 
    obtain the exact sequence \eqref{edu1}. If $M$ is strict, it coincides with \eqref{edu2}. On the contrary, let an exact sequence \eqref{edu2}
    exists. Then $M$ is faithful. We identify $A$ with $N=\setsuch{u\in nM}{\be(u)=0}$ If $a\in O(M)$, it is clear that $aN\sbe N$, whence $a\in A$.
   \end{proof}
   
   \begin{corol}\label{du5} 
   An $A$-lattice $M$ is strict \iff there is an exact sequence 
    \begin{equation}\label{edu3} 
     mM\to nM\to \om_A\to0,
    \end{equation}
     that is $M\sea \om_A$.
   \end{corol}
   \noindent
   We will also use another kind of duality analogous to the Matlis duality \cite{mat}.
   
   \begin{theorem}\label{du6} 
      Let $T_R=K\om_R/\om_R$. Denote $\hat{M}=\hom_R(M,T_R)$. The functor $M\mps\hat{M}$ induces an exact 
      duality between the categories of noetherian and artinian $R$-modules. 
    \end{theorem}
   \begin{proof}
    \textbf{Step~1}.
     We denote by $\ga_M$ the natural map $M\to\hat{\hat{M}}$.
    Any $KR$-module $V$ is an injective $R$-module and $\hom_R(V,M)=0=\hom_R(L,V)$ for any noetherian $M$ and any torsion $R$-module $L$.
    As $\idim_R\om_R=1$, $T_R$ is also an injective $R$-module. So the functor $M\mps\hat{M}$ is exact. If an $R$-module $L$ is torsion, apply 
    the functor
    $\hom_R(L,\_\,)$ to the exact sequence $0\to \om_R\to K\om_R\to T_R\to 0$. We obtain that $\hat{L}\simeq\ext^1_R(L,\om_R)$. In particular, 
    $\hat{\hat{R}}=\hat{T}_R\simeq\ext^1_R(T_R,\om_R)$. Apply to the same exact sequence the functor $\hom_R(\_\,,\om_R)$. It gives 
    $R=\hom_R(\om_R,\om_R)\simeq\ext^1_R(T_R,\om_R)=\hat{T}_R$. Thus $\ga_R$ and $\ga_{T_R}$ are isomorphisms. Now the usual observation 
    using a free presentation $mR\to nR\to M\to 0$ shows that $\ga_M$ is an isomorphism for any noetherian $R$-module $M$. 
    
    \medskip\noindent
    \textbf{Step~2}.
       Show that the module $N=\hat{M}$ is artinian if $M$ is noetherian. Indeed, if $N_1\sb N$, this embedding induces a surjection 
   $M=\hat{N}\xarr\al \hat{N_1}$ with $\ker\al\simeq\widehat{N/N_1}$. Moreover, if $N_2\sb N_1$, we have surjections 
   $M\xarr\al \hat{N_1}\xarr\be \hat{N_2}$ such that $\ker\be\al\sp\ker\al$. Thus any descending chain of submodules in $\hat{M}$
   gives an ascending chain of submodules in $M$. Therefore, there are no infinite descending chains in $\hat{M}$. In particular, the module
   $T_R=\hat{R}$ is artinian.    
    
       \medskip\noindent
    \textbf{Step~3}.
   Let now the module $N$ be artinian. It contains a simple submodule $U$. As $\hom_R(U,T_R)\ne0$ and $T_R$ is injective, there is a non-zero
   homomorphism $\al_0:N\to T_R$. As $\ker\al_0$ is also artinian, there is a non-zero homomorphism $\ker\al_0\to T_R$, which extends to a 
   homomorphism
   $\al':N\to T_R$. Let $\al_1=\smtr{\al_0\\ \al'}:N\to 2T_R$. Then $\ker\al_1\sb \ker\al_0$. Iterating this procedure, we obtain homomorphisms 
   $\al_k:N\to kT_R$ such that $\ker\al_{k+1}\sb\ker\al_k$ if $\ker\al_k\ne0$. As $N$ is artinian, there is an embedding $\be:N\to mT_R$
   for some $m$. As $\cok\be$ is also artinian, we have an exact sequence $0\to N\to mT_R\to nT_R$. Since the map $\ga_{T_R}$ is an
   isomorphism, it implies that $\ga_N$ is also an isomorphism. The observation analogous to Step~2 shows that $\hat{N}$ is noetherian, 
   which accomplishes the proof.
   \end{proof}
 \noindent
   Obviously, if we restrict this duality to $A$-modules, we obtain a duality between the categories of left (right) noetherian and right (left)
   artinian $A$-modules. One easily sees that the category of lattices is then mapped to the category of artinian modules without finite quotients.
   
   \medskip\noindent
  The duality $M\mps\hat{M}$ is closely related to the duality $D$.
  
  \begin{prop}\label{du7} 
  Let $0\to M\xarr\al N\to L\to 0$ be an exact sequence of $A$-modules, where $M,N$ are lattices and $L$ is finite. Then there is an exact
  sequence $0\to DN\xarr{D\al}DM\to\hat{L}\to 0$. In particular, if $M$ is a maximal submodule in $N$, $DN$ is a minimal overmodule of
  $DM$ and vice versa.
  \end{prop}
  \begin{proof}
   We have already seen in Step~1 of the previous proof that $\hat{L}\simeq\ext^1_A(L,\om_A)$. Note also that $\hom_A(L,\om_A)=0$. So we obtain
   the result if we apply to the given exact sequence the functor $\hom_A(\_\,,\om_A)$.
  \end{proof}
  \noindent
     Let $M$ be an $A$-lattice, $\gR=\rad A$. As $(DM)\gR$ is the intersection of maximal submodules of $DM$, its dual $M^\gR=D((DM)\gR)$
 is the sum of all minimal overmodules of $M$. If $\pi:P\xarr\pi DM$ is a projective cover of $DM$, its dual $D\pi:M\to DP$ is an inflation $\io:M\to I$ 
 such that $I$ is L-injective and $\io$ induces an isomorphism $I^\gR/I\to M^\gR/M$. We call it (and sometimes the map $\io$)
 the \emph{L-injective envelop} of $M$. We also define iterated overmodules $M^{\gR*k}$ setting $M^{\gR*1}=M^\gR$ and 
 $M^{\gR*(k+1)}=(M^{\gR*k})^\gR$. Obviously, $M^{\gR*k}=D((DM)\gR^k)$. As a principal $A$-module $P$ has one maximal submodule $\gR P$, 
 a coprincipal $A$-lattice $I$ has one minimal overmodule $I^\gR$. 

   \section{Bijective lattices and Gorenstein orders}
  \label{bi}
  
  \noindent
   Let $A$ be an $R$-order, $\gR=\rad A$. In this section we always suppose that $A$ is connected.
   
  \begin{defin}\label{bi1} 
   An $A$-lattice $B$ is called \emph{bijective} \cite{qb} if it both projective and L-injective.
  \end{defin}
  \noindent
  The main property of bijective lattices is the so called \emph{rejection lemma} (cf. \cite[Lemma~2.9]{qb}).
  
  \begin{lemma}\label{bi2} 
   Let  $B$ be a bijective $A$-lattice. Either there is a unique overring $A'$ such 
   that every $A$-lattice $M$ is isomorphic to $B'\+M'$, where $M'$ is an $A'$-lattice and $B'\Sb B$, or $A$ is hereditary and $A\Sb B$ 
   (then $M\Sb B$ for any $A$-lattice $M$).  
   
   \medskip\noindent
   \emph{We say that $A'$ is obtained from $A$ by \emph{rejection} of $B$ and denote it by  $A^-(B)$. Obviously, if $B$ is indecomposable,
   $A^-(B)$ is a minimal proper overring of $A$.}
  \end{lemma}
     
   \begin{remk}\label{bi2a} 
    By duality, $DB$ is a bijective right $A$-lattice and every right $A$-lattice $N$ is isomorphic to $B'\+N'$, where $B'\Sb B$ and $N'$
    is a right $A'$-lattice.
   \end{remk}

  \begin{proof}
    If $M\Sb B$, then $M$ is projective. Thus if $M\Sb B$ for every $A$-lattice $M$, $A$ is hereditary. So we can suppose that there are 
    lattices $M$ such that $M\not\Sb B$. Certainly, then there are faithful lattices with this property.
   If $M$ is a strict $A$-lattice, then $A\nea M$. As $B$ is projective, also $B\nea M$, whence $B\Sb M$, since $B$ is L-injective.
   Let $A'=\bigcap_M O(M)$, where $M$ runs through all faithful $A$-lattices that does not have a direct summand $B'\Sb B$. 
   There is a finite set of lattices $\lst Mn$ such that $A'=O(N)$, where $N=\bop_{i=1}^nM_i$. 
   If $N$ is strict, $B\Sb N$, which is impossible. 
   Hence $A'\sp A$ and every faithful $A$-lattice $M$ without direct summands $B'\Sb B$ is an $A'$-lattice. Let $M$ be any  
   $A$-lattice that has no direct summands $B'\Sb B$ and $\lst Us$ be all non-isomorphic simple $KA$-modules. If $M$ is not faithful, 
   at least one of them, say $U_1$, does not occur as a direct summand of $KM$. We claim that there is a $A$-lattice $L\sb U_1$ 
   such that $L\not\Sb B$. Then we replace $M$ by $M\+L$ and, continuing this procedure, obtain a faithful $A$-lattice $M'$ without 
   direct summands $B'\Sb B$ such that $M$ is its direct summand. Therefore, $M'$, hence also $M$ is an $A'$-lattice. 
   
   Suppose that $L\Sb B$ for every $A$-lattice $L\sb U_1$. Let $C$ be a simple component of $KA$ such that $U_1$ is a $C$-module,
   $A_1$ be the projection of $A$ onto $C$. If $M$ is any $A_1$-lattice, it has a chain of submodules whose factors are submodules
   of $U_1$. Hence it is projective, so $A_1$ is hereditary and is a direct factor of $A$. It implies that $A_1=A$ and $KA=C$ is a simple 
   $K$-algebra, so $M\Sb B$ for every $A$-lattice.
   \end{proof}
      
   \noindent
   To describe the structure of $A^-(B)$, we need some simple lemmas. 
    
     \begin{lemma}\label{bi3} 
    \begin{enumerate}
    \item  Let $P$ be a principal $A$-module. If all modules $\gR^iP$ are indecomposable and projective, 
    $A$ is hereditary and every indecomposable $A$-lattice is isomorphic to some $\gR^iP$.
    \item  Let $I$ be a coprincipal $A$-lattice. If all modules $I^{\gR*i}$ are indecomposable and L-injective, 
    $A$ is hereditary  and every indecomposable $A$-lattice is isomorphic to some $I^{\gR*i}$.    
    \item   Let $P$ be a principal $A$-module. If $\gR P\simeq P$, the order $A$ is maximal and $P$ is a unique
    indecomposable $A$-lattice.
    \item   Let $I$ be a coprincipal $A$-lattice. If $I^\gR\simeq I$, the order $A$ is maximal. and $I$ is a unique
    indecomposable $A$-lattice
    \end{enumerate}   
    \end{lemma}
    \begin{proof}
     (1) The conditions imply that $\gR^{i+1}P$ is a unique maximal submodule of $\gR^iP$. Therefore, any submodule of $P$
     coincides with some $\gR^iP$, hence is projective and indecomposable. Then $KP$ is a simple $KA$-module, so there is a
     simple component $C$ of the algebra $KA$ such that $KP$ is a $KA$-module. If $V$ is any $C$-module, it is a multiple of $KP$.
     Hence if $M\sb V$ is a lattice, it has a chain of submodules whose factors are submodules of $KP$. It implies that
     $M$ is projective. In particular, the projection $A_1$ of $A$ onto $C$ is projective, so is a direct summand of $A$ as $A$-module.
     It obviously implies that $A_1$ is a direct factor of $A$, so $A=A_1$.
     
     (2) is dual to (1).

     (3) If $\gR P\simeq P$, then $\gR^kP\simeq P$ for all $k$, so all of them are  principal. Just as in (1), it implies that
     $A=A_1$ and $P$ is a unique indecomposable $A$-lattice. Therefore, $A$ is maximal.
     
     (4) is dual to (3).
    \end{proof}
    
    \begin{lemma}\label{bi5} 
     We suppose that $A$ is non-hereditary.
     Let $B$ be an indecomposable bijective $A$-lattice, $A'=A^-(B)$. Then $B^\gR\niso B$ as well as $\gR B\niso B$, $B^\gR$ is 
     projective and $\gR B$ is L-injective as $A'$-lattice. 
     \end{lemma}
     \begin{proof}
     $B^\gR\niso B$ as well as $\gR B\niso B$ by Lemma~\ref{bi4}. Therefore, they are $A'$-lattices and $A'B=B^\gR$. The principal 
     $A$-module $B$ is a direct summand of $A$, so  $A\simeq B\+M$ for some $M$. Then $A'=A'A\simeq A'B\+A'M=B^\gR\+A'M$, hence 
     $B^\gR$ is projective over $A'$. By duality, $\gR B$ is L-injective over $A'$.
     \end{proof}
          
    \begin{lemma}\label{bi4} 
     \begin{enumerate}
     \item  Let $P$ be a principal  $A$-module, $M$ be its minimal overmodule. Then $M$ is either indecomposable or
     splits as $M_1\+M_2$, where $M_1,M_2$ are indecomposable. In the latter case, $\gR P=\gR M_1\+\gR M_2$ and neither $M_1$ 
     nor $M_2$ is projective. 
     \item  Let $I$ be a coprincipal $A$-lattice, $M$ be its maximal submodule. Then $M$ is either indecomposable or
     splits as $M_1\+M_2$, where $M_1,M_2$ are indecomposable. In the latter case, $I^\gR=M_1^\gR\+M_2^\gR$ and neither $M_1$ 
     nor $M_2$ is L-injective. 
     \item  Let $B$ be an indecomposable bijective $A$-lattice. Then its maximal submodule and minimal overmodule  
     decompose simultaneously. Moreover, if $\gR B$ is L-injective, $B^\gR$ is projective and vice versa.
     \end{enumerate}
    \end{lemma}
    \begin{proof}
     (1) As $P\spe\gR M\spe\gR P$, $\ell_A(M/\gR M)\le2$, hence $M$ is either indecomposable or splits as $M_1\+M_2$, where $M_1,M_2$ are 
     indecomposable. In the latter case, $\ell_A(M_1/\gR M_1)=1$, so $N=\gR M_1\+M_2\ne P$ is a maximal submodule in $M$, $N\cap P=\gR P$ 
     and $M_1/\gR M_1\simeq M/N\simeq P/\gR P$. Since $M_1\niso P$, it cannot be projective. The same applies to $M_2$. Moreover, 
     in  this case $\ell_A(M/\gR M)=2$, whence $\gR P=\gR M=\gR M_1\+\gR M_2$. 
     
     (2) follows by duality.
     
     (3) By (1) and (2), if $B^\gR$ is indecomposable, so is $\gR B$ and vice versa. Suppose that $\gR B$ is L-injective. Then it is indecomposable,
     hence $B=(\gR B)^\gR$ is a unique minimal overmodule of $\gR B$. Then $B$ is also a unique maximal submodule of $B^\gR$. 
     Therefore, there is an epimorphism $\pi:P\to B^\gR$, where $P$ is principal. If $P\simeq B$, $\pi$ is an isomorphism.
     If $P\niso B$, it is an $A'$-module. By Lemma~\ref{bi5}, $B^\gR$ is projective as $A'$-module, hence $\pi$ splits, so is an isomorphism.
     In both cases $B^\gR$ is projective over $A$. The converse follows by duality.
    \end{proof}
         
     \begin{defin}\label{link} 
      Let $B$ be a bijective $B$-lattice.
     \begin{enumerate}
     \item   A \emph{$B$-link} is a set of indecomposable lattices $\set{\lst Bl}$ such that
          \begin{itemize}
     	 \item  $B_i\Sb B$ for all $i=1,\dots,l$,
       	 \item  $B_i=\gR B_{i-1}$ for $i=2,\dots,l$ (equivalently, $B_{i-1}=B_i^\gR$),
        \item  $\gR B_l\not\Sb B$ and $B_1^\gR\not\Sb B$. 
          \end{itemize}
      \item  For an indecomposable $A$-lattice $M$ define $M^{\pm,B}$ as follows:
      		\begin{itemize}
      		\item  If $M\not\Sb B$, then $M^{\pm,B}=P$;
      		\item  If $M\in\set{\lst Bl}$, where $\set{\lst Bl}$ is a $B$-link, then $M^{+,B}=B_1^\gR$ and $M^{-,B}=\gR B_l$.
      		\end{itemize}
      We denote by $\io^B_M$ the embedding $M^{-,B}\to M^{+,B}$.
      \end{enumerate}       
    \end{defin}
                                
     \begin{theorem}\label{bi6} 
     Suppose that the order $A$ is non-hereditary. Let $B$ be a bijective $A$-lattice, $A'=A^-(B)$. 
      If $A=\bop_{i=1}^nP_i$, where $P_i$ are indecomposable, then $A'=\bop_{i=1}^nP_i^{+,B}$. In particular, all modules $P_i^{+,B}$ are 
      projective as $A'$-modules and every principal $A'$-module is isomorphic to a direct summand of some $P_i^{+,B}$. 
     \end{theorem}
     
     \begin{remk}\label{bi6a} 
      Dually, if $\om_A=\bop_{i=1}^nI_i$, where $I_i$ are indecomposable, then $\om_{A'}=\bop_{i=1}^nI_i^{-,B}$. In particular, all 
      lattices $I_i^{-,B}$ are L-injective as $A'$-lattices. 
     \end{remk}
     
     \begin{proof}
     We write $P'_i$ instead of $P_i^{+,B}$.
     Obviously, we may suppose that $B=\bop_{j=1}^mB_i$, where all $B_i$ are indecomposable and non-isomorphic. We use induction by $m$.
     Let $m=1$, i.e. $B$ is indecomposable. By Lemma~\ref{bi5}, $B^\gR\niso B$, so $B'=B^\gR$ is an $A'$-lattice and $A'B=B'$. If $P$ is  
      principal and $P\niso B$, then $P'=P$ and it is an $A'$-lattice, so $A'P=P$. Therefore, $A'=A'A=\bop_{i=1}^nP'_i$.
     
     Suppose that the theorem holds for $m-1$ summands. If $B_i^\gR\Sb B$ for all $i$, then $B_1^{\gR*k}\Sb B$ for all $k$, hence $A$ is
     hereditary by Lemma~\ref{bi3}, which is impossible. Thus we can suppose that $B_1^\gR\nSb B$. Set $A_1=A^-(B_1)$, $\gR_1=\rad A_1$.
     Then $A^-(B)=A_1^-(B(1))$, where $B(1)=\bop_{i=2}^mB_i$.
     If $\gR B_1=B_2\Sb B$, then $B_1$ is its unique minimal overmodule. As $B_1^\gR$ is a unique minimal overmodule of $B_1$ and $B_1$ is
     not an $A_1$-module, it implies that $B_2^{\gR_1}=B_1^\gR$, when we consider $A_1$-modules. Therefore, $M^{+,B}=M^{+,B(1)}$ for
     every $A_1$-lattice $M$. If $P_i\simeq B_1$ for $i\le r$ and $P_i\niso B_1$ for $i>r$, then
     $A_1=A^-(B_1)=(\bop_{i=1}^rP_i')\+(\bop_{i=r+1}^nP_i)$. Moreover, ${P'_i}^{+,B(1)}=P'_i$ for $i\le r$ and $P_i^{+,B(1)}=P_i'$ for $i>r$.
     By the induction hypothesis,  $A^-(B)=\bop_{i=1}^nP'_i$.
     \end{proof}     
     
     \noindent
    Now we introduce the class of orders which is the main in this paper. We follow the paper \cite{dkr}. The following result is an obvious corollary 
    of Propositions~\ref{du2} and \ref{du4} and Corollary~\ref{du5}.
    
   \begin{prop}\label{bi7} 
    Let $A$ be an $R$-order. \feq
    \begin{enumerate}
    \item  $A$ is L-injective as left $A$-lattice.
    \item  $A$ is L-injective as right $A$-lattice.
    \item $A\Sb M$ for every strict $A$-lattice $M$.
    \item $\om_A\Sb M$ for every strict $A$-lattice $M$.
    \item  If $M$ is a strict $A$-lattice, then $M\sea N$ for any $A$-lattice $N$.
    \item  If $M$ is a strict $A$-lattice, then $N\nea M$ for any $A$-lattice $N$.
    \item  Every projective (\fg) $A$-module is L-injective.
    \item  Every L-injective $A$-lattice is projective.
    \end{enumerate}
    \emph{If these conditions hold, $A$ is called a \emph{Gorenstein order} \cite{qb}}. 
   \end{prop} 
      
   \noindent
   Obviously, every hereditary order $A$ is Gorenstein. If $A$ is not hereditary, we denote by $A^-$ the order $A^-(A)$. It is obtained by
   rejection of all projective modules. In the Gorenstein case Theorem~\ref{bi6} can be essentially simplified using the following result.
   
   \begin{lemma}\label{bi8} 
    Let $A$ be a non-hereditary Gorenstein order, $B$ be an indecomposable bijective $A$-lattice. Then neither $B^\gR$ nor
    $\gR B$ is projective (or, the same, L-injective).
   \end{lemma}
   \begin{proof}
    Suppose that $P=B^\gR$ is projective, hence bijective. By Lemma~\ref{bi4}, it is indecomposable, hence $\gR P=B$. Let $N=P^\gR$.
    Then $\gR N\spe \gR P=B$. If $\gR N=B$, then $B^\gR\spe N$, which is impossible. Therefore, $\gR N=P$, so $N/\gR N$ is a simple
    module. Then there is a surjection $P'\to N$, where $P'$ is a principal module,
    hence a surjection $\gR P'\to P$. Thus $P$ is a direct summand of $\gR P'$. By Lemma~\ref{bi4}, $\gR P'\simeq P$, whence $P'\simeq N$,
    so $N=B^{\gR*2}$ is also bijective. Going on, we see that all lattices $B^{\gR*k}$ are bijective. By Lemma~\ref{bi3}, $A$ is hereditary,
    which is impossible, so $B^\gR$ cannot be projective. The assertion about $\gR B$ is just dual.
   \end{proof}
   
   \begin{corol}\label{bi9} 
    Let $A$ be a non-hereditary Gorenstein order, $A=\bop_{i=1}^nP_i$, where $P_i$ are indecomposable, $P'_i=P_i^\gR$
    and $B$ be a bijective $A$-lattice. Suppose that $P_i\Sb B$ for $i\le k$ and $P_i\not\Sb B$ for $i>k$. Then 
    $A^-(B)=(\bop_{i=1}^kP'_i)\+(\bop_{i=k+1}^nP_i)$. Moreover, $\gR P_i$ and $P_i^\gR$ are $A^-(B)$-lattices for all $i$. In particular, 
    $A^-=\bop_{i=1}^kP'_i$, $\gR$ and $A^\gR$ are $A^-$-lattices (both left and right).
   \end{corol}
   \begin{proof}
    Immediately follows from Theorem~\ref{bi6} and Lemma~\ref{bi8}.
   \end{proof}
\noindent
  For Gorenstein orders the rejection lemma~\ref{bi2} can be inverted.
  
  \begin{prop}\label{bi0} 
   If $A$ is Gorenstein, every minimal overring of $A$ is of the form $A^-(B)$, where $B$ is an indecomposable bijective $A$-lattice.
  \end{prop}
  \begin{proof}
   If each indecomposable projective (or, the same, bijective) $A$-lattice is actually an $A'$-lattice, then $A'=A$. Therefore, there is an
   indecomposable bijective $A$-lattice $B$ which is not an $A'$-lattice. Then $A'\spe A^-(B)$. As $A'$ is minimal, $A'=A^-(B)$.
  \end{proof}
  
  \section{Bass orders}
  \label{ba} 
  
  Recall that an order $A$ is called a \emph{Bass order} \cite{dkr} if all its overrings (including $A$ itself) are Gorenstein. The results of the
  preceding section imply the following criterion (cf. \cite[Theorem~3.1]{qb}).
  
  \begin{prop} 
   \feq
   \begin{enumerate}
   \item $A$ is a Bass order.
   \item $M\sea O(M)$ for every $A$-lattice $M$.
   \item For any two $A$-lattices $M,N$, if $M\sea N$, then $N\nea M$. 
   \item For any two $A$-lattices $M,N$, if $N\nea M$, then $M\sea N$.
   \end{enumerate}
   Therefore, any order that is Morita equivalent to a Bass order is also Bass.
  \end{prop}
  
  \begin{exam}\label{ba2} 
  \begin{enumerate}
   \item  Every hereditary order is a Bass order.
   
   \item  If every ideal of $A$ has $2$ generators, $A$ is a Bass order. It follows from \cite{roi} in the case when $R$ is a \dvr, but the proof in the
   general case is the same. 
   
   \item  Let $\De$ be a maximal order in a skewfield, $\gD=\rad\De$, $B(m,\De)$ be the subring of $\Mat(2,\De)$ consisting of such
   matrices $(a_{ij})$ that $a_{12}\in\gD^m$. It is also a Bass order (hereditary for $m=1$). We write symbolically 
   $B(k,\De)=\mtr{\De&\gD^k\\\De&\De}$.
    \end{enumerate}
\end{exam}
  
    \noindent
      Actually, it is proved in \cite{dkr} that every Bass order is either hereditary, or Morita equivalent to a local order such that every its ideal has
    $2$ generators, or Morita equivalent to some $B(\De,m)$. We will obtain this description as a corollary of
  the following result, which generalizes  \cite[Theorem~3.3]{qb}. 
   
   \begin{theorem}\label{ba3} 
    Let $A$ be a connected non-maximal order, $P$ be an indecomposable bijective $A$-lattice and $A_1=A^-(P)$. If $P^\gR\simeq\gR P$, 
    the following holds.
    \begin{enumerate}
    \item There are chains of overmodules $P=P_0\sb P_1\sb P_2\sb \dots\sb P_m$ and overrings $A=A_0\sb A_1\sb A_2\sb \dots\sb A_m$
    such that for every $0\le i< m$
          \begin{enumerate}
          \item  $P_{i+1}=P_i^{\gR_i}\simeq \gR_iP_i$, where $\gR_i=\rad P_i$. 
          \item $P_i$ is an indecomposable bijective $A_i$-lattice which is not projective over $A_{i-1}$ (hence over $A$) if $i\ne0$.
          \item $A_i$ is not maximal and $A_{i+1}=A_i^-(P_i)$.
          \end{enumerate}   
          
     \item If such a chain is of maximal length, then $A_m$ is a hereditary order, has at most $2$ non-isomorphic indecomposable lattices
          and every indecomposable $A$-lattice is isomorphic either to $P_i$ for some $0\le i< m$ or to a direct summand of $P_m$.

    \item  $A$ is Morita equivalent either to a local Bass order $E=(\End_AP)\op$ or to the Bass order $B(k,\De)$ for some $k$ and $\De$.
    \end{enumerate}
    \noindent
    The condition $P^\gR\simeq\gR P$ holds if $P^\gR$ has no L-injective summands as $A$-lattice,  but is L-injective as $A_1$-lattice, or,
    by duality, if $\gR P$ has no projective summands as $A$-module but is projective as $A_1$-module. 
   \end{theorem}
   \noindent
   Recall also that, by Lemma~\ref{bi8}, $P^\gR$ cannot have L-injective summands if $A$ is Gorenstein.
   \begin{proof}
     First, we prove the last claim.
    Theorem~\ref{bi6} implies that L-injective lattices over $A_1$ are either L-injective over $A$ or direct summands of $\gR P$. If $P^\gR$ had no
    summands L-injective over $A$ but is L-injective over $A_1$, every direct summand of $P_1$ is isomorphic to a direct summand of $\gR P$. 
    By Lemma~\ref{bi4}, either $P^\gR$ and $\gR P$ are indecomposable or $P^\gR=L_1\+L_2$ and $\gR P=\gR L_1\+\gR L_2$, 
    where $L_1,L_2,\gR L_1,\gR L_2$ are indecomposable. It implies that $P^\gR\simeq \gR P$. 
    
    \medskip\noindent
    Let $P_1=P^{\gR}\simeq\gR P$. As $A$ is not maximal, $P_1\not\simeq P$ by Lemma~\ref{bi3}.
    So there are chains of overrings and overmodules satisfying (a),(b),(c): for instance, $P=P_0\sb P_1=P^\gR$ and $A=A_0\sb A_1=A^-(P)$.
	As there are no infinite chains of overrings, consider a chain of maximal length $m$ with these properties.
  Lemma~\ref{bi5} and Theorem~\ref{bi6} imply that 
  \begin{itemize}
  \item $P_m$ is a bijective $A_m$-module, but is not projective over $A_{m-1}$ (hence over $A$) if $i\ne0$.
  \item If $i<m$, every indecomposable $A$-module either is isomorphic to one of the modules $\lsto Pi$ or is an $A_{i+1}$-module.
  \item Every principal $A_i$-module is either projective over $A$ or isomorphic to a direct summand of $P_i$ 
  (hence to $P_i$ if $i<m$).
  \end{itemize}  
  If $i<m$, $P_{i-1}\ne\gR_iP_i$, since $P_{i-1}$ is not an $A_i$-lattice, but $\gR_iP_i\spe\gR_{i-1}P_{i-1}$. If 
  $\gR_iP_i=\gR_{i-1}P_{i-1}\simeq P_i$, $A_i$ is maximal, which is impossible. Therefore, $\gR_iP_i\cap P_{i-1}=\gR_{i-1}P_{i-1}$
  and $\gR_iP_i+P_{i-1}=P_i$, whence 
  \begin{equation}\label{quo1} 
    P_i/\gR_iP_i\simeq P_{i-1}/\gR_{i-1}P_{i-1}\simeq P_{i-2}/\gR_{i-2}P_{i-2}\simeq\dots\simeq P/\gR P.
  \end{equation}
	As $\gR_iP_i\simeq P_{i+1}$ and $\gR_{i-1}P_{i-1}\simeq P_i$, we also have that
   \begin{equation}\label{quo2} 
    P_{i+1}/P_i\simeq P_i/P_{i-1}\simeq P_{i-1}/P_{i-2}\simeq\dots\simeq P_1/P
   \end{equation}
   Suppose first that $A_m$ decomposes: $P_m=L_1\+L_2$, where $ L_1$ and $L_2$ are indecomposable and non-projective over $A_{m-1}$
   (hence over $A$) by Lemma~\ref{bi4}. As $\gR_{i-1}P_m=\gR_{i-1}L_1\+\gR_{i-1}L_2\simeq L_1\+L_2$ and $\gR_{i-1}L_1,\gR_{i-1}L_2$ are
   indecomposable, either $\gR_{i-1}L_1\simeq L_1$ and $\gR_{i-1}L_2\simeq L_2$ or $\gR_{i-1}L_1\simeq L_2$ and $\gR_{i-1}L_2\simeq L_1$.
   In both cases all submodules of $L_1$ and $L_2$ are projective, isomorphic either to $L_1$ or to $L_2$. Therefore, all indecomposable 
   $A_m$-lattices are isomorphic  to $L_1$ or to $L_2$, $A_m$ is hereditary and $\lsto P{m-1},L_1,L_2$ are all indecomposable $A$-lattices.
   Hence $\lsto A{m-1}$ are all non-hereditary overrings of $A$, so $A$ is Bass. $P$ is a unique principal $A$-module,
   so $A$ is Morita equivalent to the local Bass ring $E$.
   
   Let now $P_m$ be indecomposable. Note that $P_{m-1}\spe\gR_{i-1}P_m\spe\gR_{i-1}P_{m-1}$. 
   Suppose that $P_m$ is projective as $A_{m-1}$-module.
   Then $\gR_{i-1}P_m=P_{m-1}$. Conversely, if $\gR_{i-1}P_m=P_{m-1}$, i.e. $\ell_{A_{m-1}}(P_m/\gR_{i-1}P_m)=1$, there is a surjection
   $\vi:P'\to P_m$, where $P'$ is a principal $A_{m-1}$-module. If $P'=P_{m-1}$, then $\vi$ is an isomorphism, since
   $\wh(P_{m-1})=\wh(P_m)$. Otherwise $P'$ is an $A_m$-module, thus $P'\simeq P_m$, since $P_m$ is also projective over $A_m$.
   Thus $P_m$ is projective over $A_{m-1}$, hence also over $A$.
   As $\gR_{m-1}P_m\simeq P_{m-1}$ and $\gR_{m-1}P_{m-1}\simeq P_m$, Lemma~\ref{bi3} implies that $A_{m-1}$ is hereditary and 
   $P_{m-1},P_m$ are all its indecomposable modules. Let $\De=\End_AP_m$, $\gD=\rad\De$. It is a maximal order and also
   $\End_AP_{m-1}\simeq\De$ \cite{cr}. Since $P_m\not\simeq P$, the quotients $P_m/P_{m-1}$ and $P/\gR P$ are not isomorphic. From
   \eqref{quo1} and \eqref{quo2} it follows that, for every $i<m$, $P_{i-1}$ is a unique maximal submodule of $P_i$ such that
   $P_i/P_{i-1}\simeq P_m/P_{m-1}$.  Therefore, $\vi(P_{i-1})\sbe P_{i-1}$ for every endomorphism $\vi\in\End_AP_i$, hence
   $\End_AP_i\simeq \De$ for all $i$, in particular, $\End_AP\simeq\De$. As $P$ and $P_m$ are all principal $A$-modules,
   $A$ is Morita equivalent to $\tilde{A}=\big(\End_A(P\+P_m)\big)\op$. Since any $\De$-ideal (left or right) coincides with $\gD^k$ for some $k$,
   \[
     \tilde{A}\simeq\mtr{\De&\gD^k\\\gD^l&\De}\simeq\mtr{\De&\gD^{k+l}\\\De&\De}=B(k+l,\De)
   \]
     for some $k,l$.
     
   Let now $P_m$ be indecomposable and not projective over $A_{m-1}$. Then $\gR_{m-1}P_m=\gR_{m-1} P_{m-1}$
   and $P_m\sp\gR_mP_m\spe\gR_{m -1}P_{m-1}$. If $\gR_mP_m=\gR_{m-1}P_{m-1}\simeq P_m$, then $A_m$ is maximal and $P_m$
   is a unique indecomposable $A_m$-lattice. Therefore,  $\lsto Pm$ are all indecomposable $A$-lattices, $\lsto Am$ are all overrings of
   $A$ and $A$ is Bass. Moreover, $P$ is a unique principal $A$-module and $A$ is Morita equivalent to $E$.
   
   If $P_m$ is indecomposable, not projective over $A_{m-1}$ and $P_{m-1}\ne\gR_mP_m\ne\gR_{i-1}P_{i-1}$, then $\gR_mP_m$ is a minimal 
   overmodule of $\gR_{m-1}P_{m-1}\simeq P_m$. Therefore, $\gR_mP_m\simeq P_m^{\gR_m}$, so, if we set 
   $P_{m+1}=P_m^{\gR_m},\,A_{m+1}=A_m^-(P_m)$, we obtain a longer chain of overrings and overmodules satisfying the 
   conditions (a),(b),(c), which is impossible. It accomplishes the proof.
   \end{proof}
  
   \begin{corol}[{\cite[Theorem~3.3]{qb}}]\label{ba4} 
    Let $A$ be a connected Gorenstein order. If one of its minimal overrings is also Gorenstein, then $A$ is Bass and is either hereditary, or
    Morita equivalent to a local Bass order, or Morita equivalent to an order $B(k,\De)$.
   \end{corol} 
   \begin{proof}
    It follows from Theorem~\ref{ba3} together with Lemma~\ref{bi8} and Proposition~\ref{bi0}.
   \end{proof}
   
   \begin{corol}[{\cite[Proposition~3.7]{qb}}]\label{ba5} 
    Let $A$ be a local Gorenstein order, $A'=A^-(A)$ be its minimal overring. If $A'$ is not local, $A'$ is hereditary and $A$ is Bass. 
   \end{corol}
   \begin{proof}
    By Proposition~\ref{bi0}, $A'=A^-(A)$.
    If $A'$ is not local, then $A'=P_1\+P_2$, where both $P_i$ are indecomposable projective $A'$-modules and both $\gR P_i$ are indecomposable
    L-injective $A'$-lattices. In particular, $\rad A'=\gR$. Let $P'_1$ be a minimal overmodule of $P_1$ and $M$ be a maximal submodule of $P'_1$. 
    Then $M=P_1$: otherwise, $M\cap P_1=\gR P_1$, hence $M$ is a minimal overmodule of $\gR P_1$, which is impossible,
     since $P_1$ is a unique minimal overmodule of $\gR P_1$. Thus $P_1$ is a unique maximal 
    submodule of $P'_1$, so there is an epimorphism $\vi:P\to P'_1$ for some indecomposable projective $A$-module $P$. If $P=P_1$, $\vi$ is an 
    isomorphism. If $P=P_2$, \,$\vi$ induces an epimorphism $\vi':\gR P_2\to \gR P'_1=P_1$. As $\gR P_2$ is indecomposable, $\vi'$ is an isomorphism, 
    hence so is $\vi$. Therefore, either $P'_1\simeq P_1$ or $P'_1\simeq P_2$. Just in the same way, if $P_2'$ is a minimal overmodule of $P_2$,
    then either $P_2'\simeq P_1$ or $P'_2\simeq P_2$. Now Lemma~\ref{bi3} implies that $A'$ is hereditary, thus $A$ is Bass. 
   \end{proof}

   \section{Stable categories}
   \label{st}

  \begin{defin}\label{st1} 
   \begin{enumerate}
   \item  Let $\kC$ be an additive category, $\dS$ be a set of morphisms from $\kC$. Denote by $\gnr{\dS}$ the ideal of $\kC$ generated by 
   $\dS$, i.e. consisting of morphisms of the form $\sum_{i=1}^k\al_i\si_i\be_i$, where $\si_i\in\dS$. The quotient $\kC/\gnr{\dS}$ is denoted 
   by $\kC^\dS$. Its objects are those from $\kC$ and the sets of morphisms from $M$ to $N$ are $\hom_\kC^\dS(M,N)=\hom_\kC(M,N)/\dS(M,N)$, 
   where $\dS(M,N)=\gnr{\dS}\cap\hom_\kC(M,N)$.
   
   \item The category $A\md^{\gnr{1_A}}$ is denoted by $\ula{A\md}$ and its sets of morphisms are denoted by 
   $\uhom_A(M,N)$. Obviously, it coincides with $A\md^\dP$ for $\dP=\set{1_{P_1},1_{P_2},\dots,1_{P_n}}$, where $\lst Pn$ is a complete 
   list of non-isomorphic principal $A$-modules. If $A$ is an order, the full subcategory of $A\md^{\gnr{1_A}}$ consisting of $A$-lattices coincides 
   with $A\lat^{\gnr{1_A}}$ and is denoted by $\ula{A\lat}$. We call it the \emph{stable category} of the order $A$.
   
   \item  Dually, the category $A\lat^{\gnr{1_{\om_A}}}$ is denoted by $\ola{A\lat}$ and its sets of morphisms are denoted by $\ohom_A(M,N)$. 
   Obviously, it coincides with $A\lat^\dI$ for $\dI=\set{1_{I_1},1_{I_2},\dots,1_{I_n}}$, where $\lst In$ is a complete list of non-isomorphic 
   coprincipal $A$-lattices. We call it the \emph{costable category} of the order $A$.
   \end{enumerate}
   
  \smallskip \noindent
  The duality $D$ induces a duality between the categories $\ula{A\lat}$ and $\ola{A\op\lat}$.
   If $A$ is Gorenstein, the stable and costable categories coincide.
  \end{defin}
  
  \noindent  
  Note that all $R$-modules $\uhom_A(M,N)$ and $\ohom_A(M,N)$ are of finite length. Moreover, we can estimate 
 there annihilators.
 
 \begin{lemma}\label{st0} 
  Let $A_0$ be a hereditary (for instance, maximal) overring of $A$, $\gC=\ann_R(A_0/A)$. Then $\gC^2\uhom_A(M,N)=\gC^2\ohom_A(M,N)=0$ 
  for any $M,N\in\CM(A)$.
 \end{lemma}
 \begin{proof}
  Let $M$ and $N$ be $A$-lattices, $\la,\mu\in\gC$. Consider $A_0 M\sb K M$. Then $\la A_0 M\sbe M$. 
  As $A_0$ is hereditary, $A_0M$ is a projective $A_0$-module. Hence $A_0 M$ is a direct 
  summand of a free $A_0$-module $F'$, which can be identified with $A_0 F$, where $F$ is a free $A$-module. Any homomorphism 
  $f:M\to N$ extends to a homomorphism $A_0 M\to A_0 N$, hence to a homomorphism $g:F'\to A_0 N$.  Moreover, $F\spe \la F'\spe \la M$, 
  and $\im(\mu g)\sbe\mu A_0 N\sbe N$. Therefore, the map $\la\mu f$ can be considered as the composition
  \[
   M\xarr{\ \la\ }\la M \hookrightarrow F \xarr{\,\mu g|_F\,} N.
  \]
  So $\la\mu f$ factors through a projective module and its image in $ \uhom_A(M,N)$ is zero. By duality, the same is true for $\ohom_A(M,N)$.
 \end{proof}
 
 \noindent
 There are two important functors on stable categories. Let $\pi:P\to M$ be a projective cover of a \fg\ $A$-module $M$, $\Om M=\ker\pi$. 
 Note that $\Om M$ is always an $A$-lattice, non-zero if $M$ is not projective. 
 If $M$ is a non-projective lattice, $\Om M$ is not L-injective (otherwise $\pi$ splits).
 If $\pi':P'\to M'$ is a projective cover of $M'$, any homomorphism $\al:M\to M'$ can be lifted to a homomorphism $P\to P'$, hence induces a 
 homomorphism $\ga:\Om M\to \Om M'$. If $\ga'$ comes from another lifting of $\al$, one easily checks that $\ga-\ga'$ factors through $P$. 
 Hence, the class of $\ga$ in $\ula{A\md}$ or in $\ula{A\lat}$ is well defined and $\Om$ can be considered as endofunctor on the stable category. 
 Using L-injective envelops, we can define the analogous functor $\Om'$ on $\ola{A\lat}$. If $A$ is Gorenstein, a projective cover of $M$ is also
 an L-injective envelop of $\Om M$, hence $\Om'$ is a quasi-inverse of the functor $\Om$.
 
 Let now $P_1\xarr\psi P_0\xarr\vi M\to 0$ be a \emph{minimal projective presentation} of a \fg\ $A$-module $M$, i.e. an exact sequence, 
 where $P_0,P_1$ are projective, $\ker\vi\sbe\gR P_0$ and $\ker\psi\sbe\gR P_1$. Apply to this sequence the functor 
 $\vee=\hom_A(\_\,,A)$. We obtain the exact sequence of right modules
 \begin{equation}\label{est1} 
 0\to M\ch \xarr{\vi\ch} P\ch_0\xarr{\psi\ch} P_1\ch \to \tr M\to 0,
 \end{equation}
 where $\tr M=\cok{\psi\ch}$. Again one easily checks that in this way we obtain a functor $\tr:\ula{A\md}\to\ula{A\op\md}$. As the natural map 
 $P\to P\ch{}\ch$ is an isomorphism for every \fg\ projective $P$, we have an isomorphism of functors $\b1_{\ula{A\md}}\ito\tr^2$.
 Note that if $M$ is a lattice, it can happen that $\tr M$ is not. 
 
 There is a natural map $M\ch\*_AN\to\hom_A(M,N)$, which maps $u\*v$ to the homomorphism $x\mps u(x)v$. One easily sees  \cite{ar3} that its 
 image coincides with $\dP(M,N)$. From the exact sequence \eqref{est1} it follows that $\tor^A_1(\tr M,N)\simeq\uhom_A(M,N)$.
 
 \medskip\noindent
  We will study the behaviour of $\ula{A\lat}$ and $\ola{A\lat}$ under rejection of bijective lattices. 
 
 \begin{lemma}\label{st2} 
  Suppose that the order $A$ is not maximal.
  Let $B$ be an indecomposable bijective $A$-lattice, $A'=A^-(B)$, $M,N$ be $A'$-lattices. 
  \begin{enumerate}
  \item   The natural restriction maps $\ga_+:\hom_A(B^\gR,M)\to\hom_A(B,M)$ and $\ga_-:\hom_A(M,\gR B)\to\hom_A(M,B)$ are bijective. 
  \item   A homomorphism $\al:M\to N$ factors through $B$ \iff it factors through the embedding $\gR B\to B^\gR$.
  \end{enumerate}
 \end{lemma}
 \begin{proof}
  (1) Since $B/\gR B$ is a finite module, the map $\ga_-$ is injective. Since $M$ does not have $B$ as a direct summand, $\im\al\sbe \gR B$ 
  for any $\al:M\to B$. Hence $\ga_-$ is bijective. The assertion about $\ga_+$ is just dual.
  
  (2) is an obvious consequence of (1).
 \end{proof}
 
 \begin{theorem}\label{st3} 
 Let $A$ be a non-hereditary order, $B$ be a bijective $A$-lattice, $\lst Pn$ be a complete list of non-isomorphic principal 
 $A$-modules, $\lst In$ be a complete list of non-isomorphic coprincipal $A$-lattices and $A'=A^-(B)$. Set
 $\dP^B=\setsuch{\io^B_{P_i}}{1\le i\le n}$ and $\dI^B=\setsuch{\io^B_{I_i}}{1\le i\le n}$. Then $\ula{A\lat}\simeq A'\lat^{\dP_B}$ and
 $\ola{A\lat}\simeq A'\lat^{\dI_B}$.
 \end{theorem}
 
 \noindent
 Actually, it means that, defining $\ula{A\lat}$ (respectively, $\ola{A\lat}$) we may replace $A$ by $A'$ and, for each $B$-link
 $\lst Bl$, replace in $\dP$ (respectively, in $\dI$) all maps $1_{B_i}\ (1\le i\le l)$ by the embeddings $\gR B_l\to B_1^\gR$.
 
 \begin{proof}
  If $B$ is indecomposable, the assertion follows from Lemma~\ref{st2}. Then the general case is obtained by induction on the number of
  non-isomorphic indecomposable direct summands of $B$ using Theorem~\ref{bi6}.
 \end{proof}
 
 \begin{corol}\label{st4} 
  Let $A$ be a non-hereditary Gorenstein order, $\lst Pn$ be a a complete list of non-isomorphic principal $A$-modules,
  $\io_i$ be the embedding $\gR P_i\to P_i^\gR$, $A'=A^-(A)$. Then $\ula{A\lat}\simeq A'\lat^{\dP'}$, where $\dP'=\set{\lst\io n}$.
 \end{corol}
 \begin{proof}
  It follows from Theorem~\ref{st3} and Lemma~\ref{bi9}.
 \end{proof}

 \section{Almost split sequences}
 \label{ar} 
  
  \noindent
   Recall some definitions and results (cf.~\cite{ar3}).
 Let $A$ be an order, $\al:N\to M$ and $\be:M\to N$ are homomorphisms of lattices, where $M$ is indecomposable. 
 
 \begin{defin}\label{ar1} 
  \begin{enumerate}
  \item  $\al$ is called \emph{right \asp} if the following conditions hold:
  	\begin{enumerate}
  	  \item  $\al$ is a non-split epimorphism;
  	  \item  any homomorphism $\xi:X\to M$ which is not a split epimorphism factors through $\al$;
  	  \item  if $\vi:N\to N$ is such that $\al\vi=\al$, then $\vi$ is an isomorphism.
  	\end{enumerate}  
  	\noindent 
  	Note that if (a) and (b) holds, then either (c) holds or $N=N_0\+N_1$, where $N_0\sb\ker\al$ and $\al|_{N_1}$ is right \asp.
  	
  	\medskip
  \item $\be$ is called \emph{left \asp} if the following conditions hold:
  	\begin{enumerate}
  	  \item  $\be$ is a non-split inflation;
  	  \item  any homomorphism $\xi:X\to M$ which is not a split monomorphism factors through $\be$;
  	  \item  if $\vi:N\to N$ is such that $\vi\be=\be$, then $\vi$ is an isomorphism.
  	\end{enumerate}    
  	  	\noindent 
  	Note that If (a) and (b) holds, then either (c) holds or $N=N_0\+N_1$, where $\im\be\sb N_1$ and $\be$ is left \asp\ considered as a
  	map $M\to N_1$.
  	
  	\medskip 	  
  \item  An exact sequence of $A$-lattices $\eps:0\to L\xarr\be N\xarr\al M\to 0$, where $M$ and $L$ are indecomposable, is called an 
  \emph{\asp\ sequence} if the following equivalent conditions hold:
  	\begin{enumerate}
 	 \item $\al$ is right \asp;
 	 \item  $\be$ is left \asp;
 	 \item  for every homomorphism $\xi:X\to M$, which is not a split epimorphism, the exact sequence $\eps\xi$ splits;
 	 \item  for every homomorphism $\eta:L\to X$, which is not a split inflation, the exact sequence $\eta\eps$ splits.
 	 \end{enumerate}
 	 Here $\eps\xi$ (respectively, $\eta\eps$) is the pull-back of the exact sequence $\eps$ along $\xi$ (respectively, the push-down of $\eps$
 	 along $\eta$).
  \end{enumerate}
 \end{defin}
 
 \noindent
 Obviously, a right (or left) \asp\ morphism, if exists, is unique up to isomorphism of $N$. In the same way, an \asp\ sequence with a fixed
 term $M$ (or $L$), if exists, is unique, up to isomorphism of $L$ (respectively, of $M$). Actually, in the category $A\lat$ it exists for
 every non-projective indecomposable $M$, as well as for every non-L-injective indecomposable $L$. It can be proved following literally to
 \cite{ar}. We recall the main steps.
 
 The functor $\tau_A=D\Om\tr:\ula{A\lat}\to\ola{A\lat}$ is called the \emph{\art transpose}. Just as in \cite[Proposition~1.1]{ar}, one proves that 
 \[
  \ext^1_A(N,\tau_A M)\simeq \widehat{\uhom_A(M,N)}.
 \]
 Let $M$ be an indecomposable non-projective $A$-lattice. Then the ring $\La=\uhom_A(M,M)$ is local. Dually, $\widehat{\uhom_A(M,M)}$ has a 
 unique minimal $\La$-submodule $U$. If $u$ is a non-zero element of $U$, then $u(\la)=0$ for every non-invertible $\la\in\La$. If $\xi:X\to M$ is 
 not a split epimorphism, then $\xi\vi$ is not invertible for every $\vi:M\to X$, whence $u\xi(\vi)=u(\xi\vi)=0$, i.e. $u\xi=0$. Therefore, the same holds
 for the corresponding extension $\eps\in\ext^1_A(M,\tau_A M)$, thus the extension
 \begin{equation}\label{ear1} 
   \eps: \ 0\to \tau_AM\xarr\be  E\xarr\al M\to 0,
 \end{equation} 
 is an \asp\ sequence. Note that, if $0\to L\to N\to M\to 0$ is an \asp\ sequence, so is its dual $0\to DM\to DN\to DL\to 0$. Therefore, if $L=\tau_AM$,
 then $DM\simeq \tau_ADL$ and $M\simeq D\tau_ADL\simeq \Om\tr DM$. So the functor $\tau_A$ has a quasi-inverse 
 $\tau^{-1}_A=\Om\tr D:\ola{A\lat}\to\ula{A\lat}$.
  
 Let $M=\bop_jM_j$ and $N=\bop_iN_i$, where $M_j$ and $N_i$ are indecomposable $A$-lattices. Denote by $\Rad_A(M,N)$ the set of
 homomorphisms $\vi:M\to N$ such that all components $\vi_{ij}:M_j\to N_i$ are non-isomorphisms. Obviously, we obtain an ideal of the
 category $A\lat$ called its \emph{radical}. So we can define its degrees $\Rad^n_A \ (n\in\mN)$ and $\Rad^\8_A=\bap_{n=1}^\8\Rad^n_A$.
 The homomorphisms from $\Rad_A(M,N)\setminus\Rad^2_A(M,N)$ are called \emph{irreducible}.
 The quotient ${_NV_M}=\Rad_A(M,N)/\Rad^2_A(M,N)$ is a finite dimensional vector space over the residue field $\aK$. In particular, 
 if the lattice $M$ is indecomposable, $F_M={_MV_M}$ is a skewfield, and for any $N$ both ${_NV_M}$ and ${_MV_N}$ are finite dimensional 
 vector spaces over $F_M$ (respectively, right and left). Let $A\ind$ be the set of isomorphisms classes of indecomposable $A$-lattices. The set
 $\setsuch{F_M,\,{_NV_M}}{M,N\in A\ind}$ is called the \emph{\art species} of the order $A$ and denoted by $\AR_A$. It is indeed a 
 $\aK$-species in the sense of \cite{dr}, since all $F_M$ are skewfields and ${_NV_M}$ is an $F_N\mt F_M$-bimodule. If the 
 residue field $\aK$ is algebraically closed, so $F_M=\aK$ for any indecomposable $M$, this species is usually written as a quiver 
 whose vertices are $M\in A\ind$ and there are $d_{NM}$ arrows from $M$ to $N$, where $d_{NM}=\dim_\aK({_NV_M})$. It is called 
 the \emph{\art quiver} of $A$. Obviously, the species of $A\op\lat$ is $(F\op_M,{_MV\op_N})$, where ${_MV\op_N}={_NV_M}$. So in the
 \art quiver one only has to revert all arrows.
 
 If the lattice $M$ is indecomposable and non-projective, the definition of an \asp\ sequence shows that every homomorphism from $\Rad_A(N,M)$, 
 as well as every homomorphism from $\Rad_A(\tau_AM,N)$ factors through the term $E$ of the sequence \eqref{ear1}. Hence, if
 $E=\bop_{i=1}^rE_i$ with indecomposable $E_i$, then ${_MV_N}=0={_NV_{\tau_AM}}$ if $N\niso E_i$ for all $1\le i\le r$, while
 both ${_MV_{E_i}}$ and ${_{E_i}V_{\tau_AM}}$ are non-zero. In the case of the \art quiver, there are only arrows from each of $E_i$ to $M$
 and from $\tau_AM$ to each of $E_i$. Note also that if $\al_i$ are the components of $\al$ and $\be_i$ are the components of $\be$
 in the sequence \ref{ear1}, then $\sum_{i=1}^r\al_i\be_i=0$. 
 
 If $P$ is principal, the image of any homomorphism $N\to P$, which is not a split epimorphism, belongs to $\gR P$.
 Therefore, if $\gR P=\bop_{i=1}^rE_i$ with indecomposable $E_i$, the only non-zero spaces ${_PV_N}$ are ${_PV_{E_i}}$. Dually,
 if $I$ is coprincipal and $I^\gR=\bop_{i=1}^rE_i$ with indecomposable $E_i$, the only non-zero spaces ${_NV_I}$ 
 are ${_{E_i}V_I}$.

 If the lattices $M$ and $N$ are not projective, every homomorphism from $\dP(M,N)$ is in $\Rad^2_A(M,N)$. So we can consider the 
 \emph{stable \art species} (or the \emph{stable \art quiver}) $\ula{\AR}_A$ whose objects are non-projective indecomposable lattices and the 
 bimodules ${_MV_N}$ are the same as in $\AR_A$. Dually, the \emph{costable \art species} (or \emph{costable \art quiver}) $\ola{\AR}_A$ 
 is defined, consisting of non-L-injective indecomposable lattices. The functor $\tau_A$ induces the \emph{\art translation} $\ula{AR}_A\ito\ola{AR}_A$.
 Again, in Gorenstein case stable and costable species (or quivers) coincide.
 
 We will use the following fact about irreducible morphisms between indecomposable lattices. Perhaps, it is known, though we have not found it
 in the literature. 
 
 \begin{prop}\label{ar2} 
  Let $M,N$ be indecomposable lattices, $\al:N\to M$ is an irreducible morphism. There are two possibilities:
  \begin{enumerate}
  \item $\al$ is an isomorphism of $N$ onto a direct summand of a maximal submodule of $M$.
  \item $\al$ is an epimorphism and there is a submodule $L\sb N$ such that $N/L$ is a lattice, $L+\ker\al=N$ and
  $L\cap\ker\al$ is an L-irreducible lattice.
  \end{enumerate}
 \end{prop}
 \begin{proof}
  Let $M'=\im\al$, $\io$ be the embedding $M'\to M$ and $\pi$ be the projection $N\to M'$. If $M=\bop_{i=1}^mM_i$, where $M_i$ are
  indecomposable, $\io_i$ and $\pi_i$ are the components of $\io$ and $\pi$ with respect to this decomposition. Then 
  $\al=\sum_{i=1}^m\io_i\pi_i$. As $\al$ is irreducible, at least one of $\io_i$ or $\pi_i$ must be invertible. Suppose that one of $\io_i$ is
  invertible. Then $m=1$ and $\al$ is an epimorphism. If $\ker\al$ is L-irreducible, we can set $L=\ker\al$. If $\ker\al$ is not L-irreducible,
  it contains an L-irreducible sublattice $S$ such that $N/S$ is a lattice (take the intersection of $\ker\al$ with a simple $KA$-submodule
  in $K\ker\al$). Then $\al$ factors through the map $\bar\pi:N/S\to N/\ker\al\simeq M$. Therefore, $\bar\pi$ must be a split epimorphism,
  so $N/S\simeq N/\ker\al\+N/L$ for some $L\sp S$ (in particular, $N/L$ is a lattice). It actually means that $L+\ker\al=N$ and $L\cap\ker\al=S$, 
  which gives the possibility (2).
  
 If one of $\pi_i$ is invertible, then all other $\pi_j=0$ and $\al$ is a monomorphism. If $M'$ is a maximal submodule of $M$ containing $\im\al$,
 then $\al$ factors through the embedding $\im\al\to M'$, hence the latter must split. It gives the possibility (1).
 \end{proof}
 
 \noindent
 We study the behaviour of these constructions under rejection of bijective lattices. First, a simple observation.
 
 \begin{prop}\label{ar3} 
 Let $B$ be a bijective $A$-lattice, $A'=A^-(B)$, $M$ be an $A'$-lattice.
 \begin{enumerate}
 \item  If $\al:N\to M$ is right \asp\ in $A'\lat$, it is so in $A\lat$. 
 \item  If $\be:M\to N$ is left \asp\ in $A'\lat$, it is so in $A\lat$. 
 \item  If $0\to L\to M\to N\to 0$ is an \asp\ sequence in $A'\lat$, it is so in $A\lat$.  
 \end{enumerate}
  \end{prop}
 \begin{proof}
  (1) Let $\xi\in\hom_A(X,M)$ be not a split epimorphism. If $X\not\Sb B$, it is an $A'$-lattice, so $\xi$ factors through $\al$. If $X\Sb B$, it is
  projective, so $\xi$ also factors through $\al$.
  
  (2) by duality.
  
  (3) follows from (1) or (2).
 \end{proof}
 
 \noindent
 The following theorem describes the ``\art behaviour'' of new projective modules over the order $A^-(B)$.
 
 \begin{theorem}\label{ar4} 
  Let $B$ be an indecomposable bijective $A$-lattice, $A'=A^-(B)$. Suppose that $B^\gR$ is not projective over $A$ (equivalently, $\gR B$
  is not L-injective over $A$). 
  \begin{enumerate}
  \item  If $B^\gR$ decomposes: $B^\gR=M_1\+M_2$, there are \asp\ sequences
  \begin{align*}
    & 0\to \gR M_1\to B\to M_2\to 0,\\
    & 0\to \gR M_2\to B\to M_1\to 0.
  \end{align*}
  In particular, $\tau_AM_1=\gR M_2$ and $\tau_AM_2=\gR M_1$.
  
  \smallskip
  \item If $B^\gR$ is indecomposable, then $\tau_AB^\gR=\gR B$, \,$B^\gR$ has a maximal submodule $X\ne B$ and there is an 
  \asp\ sequence
  \begin{equation}\label{ear2} 
     0\to \gR B\to B\+X \xarr\al B^\gR\to 0.
  \end{equation}
  In particular, $\tau_AB^\gR=\gR B$.
  \end{enumerate}
 \end{theorem}  
 \begin{proof}
  $B^\gR$ is projective and $\gR B$ is L-injective over $A'$ by Lemma~\ref{bi5}. Let $M$ be a direct summand of $B^\gR$, $N=\tau_AM$ and
  $0\to  N\to E\to M\to 0$ be an \asp\ sequence in $A\lat$. If $N$ were not L-injective as $A'$-lattice, there were an \asp\ sequence
  $0\to N\to E'\to M'\to 0$ in $A'\lat$. By Proposition~\ref{ar3}, it were also an \asp\ sequence in $A\lat$, whence
  $M'\simeq M$, which is impossible, since $M$ is projective over $A'$. Thus $\tau_AM$ is L-injective as $A'$-lattice, but not as $A$-lattice.
  Therefore, it is a direct summand of $\gR B$. In particular, if $B^\gR$ is indecomposable, $\tau_AB^\gR=\gR B$. 
  
   There is an irreducible morphism $B\to M$, hence $B$ must be a direct summand of $E$, so $E=B\+X$. If $B^\gR=M_1\+M_2$, there is an
   exact sequence $0\to \gR M_1\to B\to M_2\to 0$. As $KB\simeq KM_1\+KM_2$, $X=0$. If $B$ is indecomposable, $KX\simeq KB$. Hence
   Proposition~\ref{ar2} implies that in the \asp\ sequence \eqref{ear2} the restriction of $\al$ on $X$ is an isomorphism onto a maximal 
   submodule of $B^\gR$ which cannot coincide with $B$. 
 \end{proof}
 
 \begin{remk}\label{ar5} 
 \begin{enumerate}
 \item   It can happen that in case (1) $M_1\simeq M_2$\, and in case (2) $X\simeq B$. If $X\niso B$, then it is an $A'$-lattice and $X=\gR'B^\gR$,
  where $\gR'=\rad A'$. If $X\simeq B$, then $\gR'B^\gR=\gR B^\gR$.
 \item  By Lemma~\ref{bi8}, the condition ``$B^\gR$ is not projective'' always holds if $A$ is connected, Gorenstein and non-hereditary.
 \end{enumerate}
 \end{remk}

 \section{Gorenstein and Frobenius cases}
 \label{go}
 
 If the order $A$ is Gorenstein, the functor $\vee:M\mps M\ch=\hom_A(M,A)$ is an exact duality $A\lat\to A\op\lat$. Combining it with the duality
 $D:A\op\lat\to A\lat$, we obtain the \emph{Nakayama equivalence} $\kN=D\vee:A\lat\to A\lat$. It maps projective modules to
 projective, thus can also be considered as the functor on stable categories $\ula{A\lat}\to\ula{A\lat}$. The following result is an analogue
 of \cite[Proposition~IV.3.6]{ar3}.
 
 \begin{prop}\label{go1} 
  If the order $A$ is Gorenstein, the functors $\tau_A,\,\Om\kN$ and $\kN\Om$ are isomorphic.
 \end{prop}
 \begin{proof}
  Let $M$ be a non-projective $A$-lattice.
  Consider an exact sequence $$0\to N\xarr\al P_1\xarr\be P_0\xarr\ga M\to 0,$$ where $P_1\xarr\be P_0\xarr\ga M\to 0$ is a minimal
  projective presentation of $M$. It gives the exact sequence 
  \[
   0\to M\ch \xarr{\ga\ch} P_0\ch \xarr{\be\ch} P_1\ch \xarr{\al\ch} N\ch\to 0.
  \]
  Thus $N\ch\simeq\tr M$ and $\Om\tr M\simeq\im\be\ch$. Now the exact sequence 
  \[
   0\to D(\im\be\ch) \to P_0^{\vee\vee} \to DM\ch \to0 
  \]
  shows that $\tau_AM\simeq D(\im\be\ch)\simeq \Om\kN M$. One easily sees that this construction is functorial in $M$, so it gives an isomorphism
  $\tau_A\simeq \Om\kN$. Since $\kN$ is exact and maps projective modules to projective, it commutes with $\Om$, i.e. $\Om\kN\simeq\kN\Om$.
 \end{proof}
 
 Let $A\simeq\bop_{i=1}^sm_iP_i$, where $\lst Ps$ are pairwise non-isomorphic principal left $A$-modules. Then also $A\simeq\bop_{i=1}^sm_iP\ch_i$
 as right $A$-module, $A\simeq\bop_{i=1}^sm_iDP\ch_i$ as left $A$-module, and $\lst{DP\ch}s$ are all pairwise non-isomorphic coprincipal
 left $A$-modules. Therefore, $A$ is Gorenstein \iff  there is a permutation $\nu$ such that $P_i\simeq DP\ch_{\nu i}$ for all $i=1,2,\dots,s$.
 The permutation $\nu$ is called the \emph{Nakayama permutation}.
 
 \begin{defin}\label{go2} 
  An order $A$ is called \emph{Frobenius} if $A\simeq DA$ as left $A$-module. It is called \emph{symmetric} if $A\simeq DA$ as $A$-bimodule.
 \end{defin}

 Obviously, $A$ is Frobenius \iff it is Gorenstein and $m_i=m_{\nu i}$ for all $i=1,2,\dots,s$, where $\nu$ is the Nakayama permutation. One
 easily sees that in this case also $A\simeq DA$ as right $A$-module, so the definition of Frobenius orders is left-right symmetric.
 
 \begin{defin}\label{go3}  
  Let $M$ be a left $A$-module, $\si$ be an automorphism of $A$. We denote by ${^\si\!M}$ the left $A$-module such that it coincides with 
  $M$ as a group, but, for every $a\in A$ and $x\in M$, the product $ax$ in ${^\si\!M}$ coincides with the product $\si(a)x$ in $M$. Analogously $N^\si$
  is defined for a right $A$-module $N$ and ${^\rho\!M^\si}$ is defined for an $A$-bimodule $M$, where $\rho$ is also an automorphism of $A$. 
  If $\rho$ or $\si$ are identity, it is omitted and we write, respectively, $M^\si$ or ${^\rho\!M}$.
 \end{defin}
 
 One easily sees that the maps $x\mps\rho^{-1}(x)$ and $x\mps\si^{-1}(x)$ give isomorphisms of $A$-bimodules, respectively, 
 ${^\rho\!A^\si}\simeq A^{\rho^{-1}\si}$ and ${^\rho\!A^\si}\simeq {^{\si^{-1}\rho}\!A}$.
 
 \begin{prop}\label{go4} 
  $A$ is Frobenius \iff there is an automorphism $\si\in\aut A$ such that $DA\simeq A^\si$ as $A$-bimodule. Moreover, there is an invertible element
  $s\in KA$ such that $\si(a)=s^{-1}as$ for all $A$. 
 \end{prop}
 \begin{proof}
  Obviously, if such an automorphism exists, $A$ is Frobenius. Suppose that $A$ is Frobenius and let $\vi:A\ito\De$ be an isomorphism of left 
  $A$-modules, 
  where $\De=DA$. It induces an isomorphism of left $KA$-modules $K\vi:KA\ito K\De$. Since $KA$ is semisimple, it is symmetric as $K$-algebra,
  \cite[9.8]{cr} i.e. there is an isomorphisms of $KA$-bimodules $\th:KA\ito K\De$. The composition $\th^{-1}{\cdot} K\vi$ is an automorphism of $KA$ 
  as of left $KA$-module, hence there is an invertible element $s\in KA$ such that $\th^{-1}{\cdot} K\vi(x)=xs$ for every $x\in KA$. In particular,
  $\vi(x)=\th(xs)$ for every $x\in A$, so $\De=\th(As)$. It implies that $As=\th^{-1}(\De)$ is a two-sided $A$-module, so $sA\sbe As$ or
  $sAs^{-1}\sbe A$. Therefore, $sAs^{-1}=A$ and $s^{-1}As=A$. Moreover,
  \[
   \vi(xa)=\th(xas)=\th(xs\cdots^{-1}as)=\th(xs)s^{-1}as=\vi(x)s^{-1}as.
  \]
 Hence $\vi$ is an isomorphism of $A$-bimodules $A^\si\ito\De$, where $\si(a)=s^{-1}as$.
 \end{proof}
 
 \noindent
 One can check that the element $s$ above is defined up to a multiplier of the form $q\la$, where $q$ and $\la$ are invertible element, respectively,
 from $A$ and from the center of $KA$.
 
 \begin{corol}\label{go5} 
  Let $A$ be a Frobenius order, $\si\in\aut A$ be as in Proposition~\ref{go4}, $\kN$ be the Nakayama equivalence. There are functorial
  isomorphisms:
  \begin{itemize}
    \item   $DM\simeq(M\ch)^\si$ for any left $A$-lattice $M$ and $DN\simeq{^{\si^{-1}}\!(N\ch})$ for every right $A$-lattice $N$;
   \item   $\kN M\simeq {^{\si^{-1}}\!M}$ and $\tau_AM\simeq \Om({^{\si^{-1}}\!M})\simeq{^{\si^{-1}}\!(\Om M)}$ for every left $A$-lattice $M$.
    \end{itemize}  
  In particular, if $A$ is symmetric, $\kN\simeq\id$ and $\tau_A\simeq\Om$.
 \end{corol}
 \noindent\emph{Proof} is obvious. \qed

 \begin{corol}\label{go6} 
 Let $A$ be a Gorenstein order, $\gR=\rad A$, $\lst Ps$ be a complete set of non-isomorphic principal $A$-modules, $\om_i=DP_i\ch$ (then
 $\lst\om s$ is a complete set of non-isomorphic coprincipal $A$-modules). Set $A'=A^-(A)$, $P'_i=P_i^\gR$ and $\om'_i=\gR\om_i$. Then
 $\tau_AP'_i\simeq \om'_{\nu i}$, where $\nu$ is the Nakayama permutation.
 \end{corol}
 \begin{proof}
  It follows from Theorem~\ref{ar4}.
 \end{proof}
    
 \begin{corol}\label{go7} 
 Let $G$ be a finite group, $A$ be a block of the group ring $\mZ_pG$. It is a symmetric $\mZ_p$-order. Set $A'=A^-(A)$. Then, for every 
 non-projective $A$-lattice $M$ (or, the same, for every $A'$-lattice $M$),
 \[
  \hH^n(G,M)\simeq \hH^{n+1}(G,\tau_AM)\simeq \hH^{n-1}(G,\tau_A^{-1}M).
 \] 
 \end{corol}
 \begin{proof}
  It follows from Corollary~\ref{go5} and Proposition~\ref{ar3}.
 \end{proof}
 
 Note that $\tau_AM=\tau_{A'}M$ if $M$ is not projective over $A'$. Otherwise $\tau_AM$ is given by Corollary~\ref{go6}. In some cases the structure
 of the \art species $AR_{A'}$ can be calculated explicitly. Then it gives the values of the cohomologies. An example, when $G$ is the Kleinian 4-group,
 can be found in \cite{dp}.

\end{document}